\documentclass[12pt]{amsart}
\usepackage{bbm}
\usepackage{mathrsfs}
\usepackage[english]{babel}
\usepackage{wrapfig}

\usepackage[top=3cm,bottom=2cm,left=3cm,right=3cm,marginparwidth=1.75cm]{geometry}

\usepackage{amsmath,amssymb}
\usepackage{graphicx}
\usepackage{cite}
\usepackage{enumitem} 
\usepackage[all,cmtip]{xy}
\usepackage{tikz-cd}
\usetikzlibrary{intersections}
\usepackage{relsize}
\usepackage{faktor}
\usepackage{dsfont}
\usepackage{xcolor, import}
\usepackage{bm}
\usepackage{hyperref}
\hypersetup{
 citebordercolor={green!40!black},
 linkbordercolor={red!50!black},
 urlbordercolor={blue!70!black}
}

\makeatletter
\renewcommand{\section}{\@startsection%
{section}
{1}
{0mm}
{1.5\bigskipamount}
{0.5\bigskipamount}
{\centering\normalsize\sc}}

\renewcommand{\subsection}{\@startsection%
{subsection}
{2}
{0mm}
{0.5\bigskipamount}
{0.5mm}
{\normalsize\sc}}

\renewcommand{\paragraph}{\@startsection%
{paragraph}
{4}
{0mm}
{\bigskipamount}
{0pt}
{\normalsize\bf}}

\expandafter\let\expandafter\oldproof\csname\string\proof\endcsname
\let\oldendproof\endproof
\renewenvironment{proof}[1][\proofname]{%
  \oldproof[\slshape #1]%
}{\oldendproof}
\def\provedboxcontents#1{$\square$}

\newtheoremstyle{thm}{6pt plus 1pt minus 1pt}{6pt plus 1pt minus 1pt}{\slshape}{}{\scshape}{.}{5pt plus 1pt minus 1pt}{}
\newtheoremstyle{def}{6pt plus 1pt minus 1pt}{6pt plus 1pt minus 1pt}{}{}{\scshape}{.}{5pt plus 1pt minus 1pt}{}
\newtheoremstyle{rmk}{6pt plus 1pt minus 1pt}{6pt plus 1pt minus 1pt}{}{}{\scshape}{.}{5pt plus 1pt minus 1pt}{}
\newtheoremstyle{claim}{6pt plus 1pt minus 1pt}{6pt plus 1pt minus 1pt}{}{}{\slshape}{.}{5pt plus 1pt minus 1pt}{}

\theoremstyle{thm}
\newtheorem{newstatement}{newstatement}
\newtheorem{lemma}[newstatement]{Lemma}
\newtheorem{theorem}[newstatement]{Theorem}
\newtheorem*{theorem*}{Theorem 2}
\newtheorem{corollary}[newstatement]{Corollary}
\newtheorem{proposition}[newstatement]{Proposition}

\newtheorem*{reeb-thm}{Reeb's Sphere Theorem}

\theoremstyle{def}
\newtheorem{definition}[newstatement]{Definition}
\newtheorem{question}[newstatement]{Question}

\theoremstyle{rmk}
\newtheorem{remark}[newstatement]{Remark}
\newtheorem{example}[newstatement]{Example}

\theoremstyle{claim}

\newcommand{\supl}[2]{A_{#2}(#1)}
\newcommand{\subl}[2]{B_{#2}(#1)}

\let\grad\nabla

\renewcommand{\epsilon}{\varepsilon}
\renewcommand{\phi}{\varphi}

\renewcommand{\:}{\,{:}\;}
\DeclareMathOperator{\st}{st}

\usepackage{mathtools} 
\mathtoolsset{showonlyrefs,showmanualtags}
\numberwithin{equation}{section}

\newcommand{\R}{\mathbb{R}}
\newcommand{\C}{\mathbb{C}}

\newcommand{\Z}{\mathbb{Z}}

\newcommand{\Bd}{\partial}
\DeclareMathOperator{\Int}{Int}

\newcommand{\be}{\begin{equation}}
\newcommand{\ee}{\end{equation}}

\let\emph\textsl
\title{On smooth functions with two critical values}

\author{Antonio Lerario}
\address{SISSA, Via Bonomea 265, 34136 Trieste, Italy}
\email{lerario@sissa.it}

\author{Chiara Meroni}
\address{Max Planck Institute for Mathematics in the Sciences, Inselstrasse 22, 04103 Leipzig, Germany}
\email{chiara.meroni@mis.mpg.de}

\author{Daniele Zuddas}
\address{Dipartimento di Matematica e Geoscienze, Università di Trieste, Via Valerio 12/1, 34127 Trieste, Italy.}
\email{dzuddas@units.it}

\begin{document}

\begin{abstract}
We prove that every smooth closed manifold admits a smooth real-valued function with only two critical values. We call a function of this type a \emph{Reeb function}. We prove that for a Reeb function we can prescribe the set of minima (or maxima), as soon as this set is a PL subcomplex of the manifold. In analogy with Reeb's Sphere Theorem, we use such functions to study the topology of the underlying manifold. In dimension $3$, we give a characterization of manifolds having a Heegaard splitting of genus $g$ in terms of the existence of certain Reeb functions. Similar results are proved in dimension $n\geq 5$.
\end{abstract}



\maketitle

\section{Introduction}\label{sec:intro}

In 1946 Reeb \cite{Reeb46} proved the following result.

\begin{reeb-thm}
Let $M$ be a smooth closed manifold of dimension $n$. Suppose that there is a smooth function $f \: M \to \R$ with only two critical points, which are also non-degenerate. Then $M$ is homeomorphic to the $n$-sphere.
\end{reeb-thm}

Subsequently, Milnor \cite{Milnor64} improved the theorem by relaxing the non-degeneracy assumption for the two critical points.

In the present article we study smooth functions on closed manifolds having two critical values (instead of two critical points), which are then the minimum and the maximum, respectively. Such functions (which are defined below, see Definition \ref{Reeb_funct}) will be called \textsl{Reeb functions}. Moreover, we will be able to keep some control on the critical subset of a Reeb function. Some examples of Reeb functions are provided at the beginning of Section \ref{sec:Reeb}.

It is a natural question to understand to what extent the knowledge of the set $X_0$ of minima and $X_1$ of maxima of a Reeb function $f\: M\to \R$ determines the manifold $M$ (at least up to homeomorphisms). This problem is already interesting when $X_0, X_1$ are smooth submanifolds of $M$ and, to our knowledge, it does not seem to have been investigated in the literature. In order to address this question, in this paper, we will consider Reeb functions whose critical sets $X_0$ and $X_1$ are finite subcomplexes of $M$. Notice that this requirement imposes restrictions only on the structure of the critical set, and not on the behaviour of the function near it (i.e., we allow $f$ to assume the minimum and the maximum values with arbitrary order of degeneracy).

To start with, we prove that every smooth closed manifold $M$ admits a Reeb function, even with a prescribed set of minima (or maxima).

\begin{theorem}\label{thm:existence}
Let $M$ be a smooth closed manifold, and let $X \subsetneq M$ be a finite simplicial subcomplex.
Then there exists a $C^{\infty}$ Reeb function $f\: M \to [0,1]$, having $0$ and $1$ as extrema, such that $f^{-1}(0)= X$.
\end{theorem}
The homology of maxima and minima of a Reeb function satisfy some ``duality'' relation, see Remark \ref{remark:homology}. For example, if $M=S^n$ then the homology of $X_0$ is isomorphic to the homology of $X_1$, with reversed indices. Moreover, as one could expect, $M\setminus X_1$ deformation retracts to $X_0$ and viceversa, see Lemma \ref{lemma:h}.

We can restate Reeb's Sphere Theorem using the language of Reeb functions  as follows: a smooth closed manifold of dimension $n$ admits a smooth Reeb function with two points as extrema, if and only if it is homeomorphic to an $n$-sphere. What happens if the set of minima and maxima are subcomplexes? In dimension $n=3$ we prove the following result.

\begin{theorem}\label{thm_g_infty}
Let $M$ be a smooth closed connected $3$-manifold. Then the following conditions are equivalent:
\begin{enumerate}
\item there exists a Reeb function $f \colon M \to [0,1]$  having connected 1-dimensional subcomplexes of the same Euler characteristic $\chi$ as extrema, for some $\chi\leq 1$;
\item the Heegaard genus of $M$ is at most $1-\chi$ in the orientable case and $2-2\chi$ in the non-orientable case.
\end{enumerate} 
\end{theorem}

Results in the same spirit as of Theorem \ref{thm_g_infty} can be  proved in higher dimensions using cobordism theory, see Corollary \ref{cor:main_higher_dimension}.

\begin{remark}
The Reeb graph of a Reeb function with connected extrema is the connected graph with two vertices and one edge (that is, an interval), see Saeki \cite{S2022}. Therein, smooth functions with given Reeb graph on a given closed manifold have been constructed under a certain compatibility condition expressed in terms of fundamental groups, and the critical set turns out to be an embedded smooth manifold of codimension zero whose connected components correspond to the vertices of the graph.
\end{remark}

\section{Reeb's functions}\label{sec:Reeb}
In the following, smooth means $C^\infty$, and diffeomorphisms between smooth manifolds are $C^\infty$ unless otherwise stated.

\begin{definition}\label{Reeb_funct}
Let $M$ be a smooth closed manifold and let $f\: M \to [a,b]$ be a $C^k$ function, with $k\geq 2$ and $a<b$. We say that $f$ is a \emph{Reeb function} on $M$ if $f$ has $a$ and $b$ as the only critical values (the \textsl{extrema} of $f$). The subsets $f^{-1}(a)$ and $f^{-1}(b)$ are the \emph{critical levels} of $f$.\end{definition}

Without loss of generality, we assume throughout the paper that $a=0$ and $b=1$.
The simplest example of a Reeb function is a function on $S^n$ with only two critical points. Below we give some other examples.
\begin{example}Let $M$ be a smooth manifold and $g\: S^n\to [0,1]$ be a smooth function with only two critical points $p_0, p_1\in S^n$. Then the function $f\: M \times S^n\to [0,1]$ defined by $f(x, y) =g(y)$ is a Reeb function with $X_0=M \times \{p_0\}$ and $X_1=M \times \{p_1\}.$
\end{example}

\begin{example}Let $\mathbb{K}\!=\!\R, \C, \mathbb{H}$ and denote by $\mathbb{K}\mathrm{P}^n$ the $n$-dimensional projective space  over $\mathbb{K}$ with homogeneous coordinates $[x_0, \ldots, x_n]$. Let $X_0\!=\!\{x_i=0 \mid i\leq k\}\cong \mathbb{K}\mathrm{P}^{n-k-1}$ and $X_1=\{x_i=0 \mid i> k\}\cong \mathbb{K}\mathrm{P}^k$. The function $f\: \mathbb{K}\mathrm{P}^n\to [0,1]$ defined by
$$ f(x)=\frac{|x_0|^2+\cdots +|x_k|^2}{|x_0|^2+\cdots +|x_n|^2}$$
is a Reeb function (in fact, also a Morse--Bott function).
\end{example}

\begin{example}
Let  $M\subset \R^n$ be a smooth semialgebraic set and  $X\subset M$ be a basic semialgebraic set (relative to $M$). There is an explicit way to define a Reeb function $f\: M\to \R$ with $f^{-1}(0)=X$. In fact, assuming that
\begin{equation}
    X = \{x \in \R^n \,|\, g_1(x)\geq 0, \ldots , g_s(x)\geq 0\}\cap M
\end{equation}
where $g_i \in \R[x_1,\ldots ,x_n]$ for every $i=1,\ldots,s$, the function $f$ can be constructed as follows. First consider the functions $f_i\: \R^n \to \R$ such that
\begin{equation}
    f_i (x) = 
    \begin{cases}
    0 & g_i(x) \geq 0, \\
    e^{\frac{1}{g_i(x)}} & g_i(x) < 0.
    \end{cases}
\end{equation}
Then set $f= \sum_{i=1}^s f_i|_{M}$. This example is also connected to the theory developed by Durfee  \cite{durfee:nhoods} and Dutertre \cite{dutertre:semial_nhoods}.
\end{example}

\begin{example}
If $f:M\to [0,1]$ is a Reeb function, then for every integer $n>0$ the function $f^n$ is also Reeb, with the same critical set of $f$ but with a higher order of degeneracy at zero. More generally, if $g:I\to I$ is a smooth map such that $g'(t)\neq 0$ for all $t\in (0,1)$, then $g\circ f:M\to [0,1]$ is still a Reeb function.
\end{example}

Throughout the paper, we make use of the following notation. Given a function $f \: M \to \R$ we set
\begin{align*}
   \supl{f}{\epsilon} &= \{x\in M \mid f(x)\geq \epsilon \} && \text{(the $\epsilon$-superlevel of $f$)}\\
   \subl{f}{\epsilon} &= \{x\in M \mid f(x)\leq \epsilon \} && \text{(the $\epsilon$-sublevel of $f$).}
\end{align*}

We recall the following well-known fact, which we state without proof (the reader may refer to Lee \cite{lee:smooth_man} for a proof).

\begin{lemma}\label{remark:glue}
Let $M_0$, $M_1$ be two smooth $m$-manifolds with non-empty boundary, and suppose that there is a diffeomorphism $\phi \: \Bd M_0 \to \Bd M_1$ between their boundaries. Let $f_0 \: M_0 \to [a,b]$ and $f_1 \: M_1 \to [b,c]$ be $C^k$ functions, $1\leq k \leq \infty$,
such that $f_0^{-1}(b) = \Bd M_0$ and $f_1^{-1}(b) = \Bd M_1$. Furthermore, we assume that $\grad f_0$ is non-vanishing and pointing outwards along $\Bd M_0$, and $\grad f_1$ is non-vanishing and pointing inwards along $\Bd M_1$, with respect to certain Riemannian metrics on $M_0$ and $M_1$. Then, the manifolds $M_0$ and $M_1$, as well as the functions $f_0$ and $f_1$, can be glued together by $\phi$ yielding a smooth manifold $M= M_0\cup_\phi M_1$ and a $C^k$ function $f = f_0 \cup_\phi f_1 \: M \to [a,c]$. Moreover, this construction is unique up to $C^k$ diffeomorphisms of $M$.
\end{lemma}

For the proof of Theorem \ref{thm:existence} we will need the following two technical lemmas.

\begin{lemma}[Existence of arbitrarily flat functions]\label{lemma:smooth}Let $\{c_k\}_{k\in \mathbb{N}}$ be a sequence of positive real numbers. There exists a smooth function $\gamma\: [0, 1]\to \R$ such that $\gamma(0)=0$, $\gamma(1)>0$, $\gamma'(t)>0$ for $t\in (0,1]$ and for every $k\in \mathbb{N}$ 
\be\sum_{j=0}^k|\gamma^{(j)}(t)|\leq c_k\quad \forall t\leq \frac{1}{k}.\ee
\end{lemma}
\begin{proof}
Without loss of generality, we assume that the sequence $\{c_k\}_{k\in \mathbb{N}}$ is monotone decreasing with $c_0<1$.
Consider the smooth function $\varphi\: [0, 1]\to [0,1]$ defined by:
\be \varphi(x)=\frac{e^{\frac{-1}{x}}}{e^{\frac{-1}{x}}+e^{\frac{-1}{1-x}}}.\ee
This is a monotone function on $[0,1]$ with $\varphi^{(k)}(0)=0=\varphi^{(k)}(1)$ for all $k\geq 0.$ For every $k\geq1$ let us define the function $\varphi_k\: [\frac{1}{k+1}, \frac{1}{k}]\to \R$ by
\be \varphi_k(x)=a_k\varphi(k(k-1)x-k),\ee
where the sequence $\{a_k>0\}_{k\geq1}$ is chosen in such a way that
\be \sup_{t\in [\frac{1}{k+1},\frac{1}{k}]}\sum_{j=0}^{k-2}|\varphi_k^{(j)}(t)|\leq \frac{c_{k}}{3}.\ee
By replacing $a_k$ with a possibly smaller positive number, we can assume that:
\begin{enumerate}\item $\{a_k\}_{k\geq 1}$ is monotone decreasing and that $\sum_{k\geq 1}a_k=A<1.$ 
\item for every $k\geq 1$ we have $\sum_{j\geq k+1}a_j\leq \frac{c_k}{3}$;
\item denoting by $\sigma=(\sum_{j}\frac{1}{j^4})^{1/2}$, for every $k\geq 1$ we have $(\sum_{j\geq k}a_j^2)^{1/2}\leq \frac{c_k}{3\sigma }.$
\end{enumerate} 
Pick now the sequence $\{b_k\}_{k\geq 1}$ defined by $b_k=b_{k-1}-a_k$ and $b_1=1-a_1$, and define $f\: [0, 1]\to \R$ to be the function
\be f(t)=\left(\sum_{k\geq 1}(\varphi_k(t)+b_k)\chi_{[\frac{1}{k+1}, \frac{1}{k}]}(t)\right)-(1-A).\ee
This function is smooth, since each $\varphi_k$ is flat at its domain of definition and at $t=\frac{1}{k+1}$ we have $\varphi_k(\frac{1}{k+1})+b_k=\varphi_{k+1}(\frac{1}{k+1})$. Moreover, the function $f$ satisfies $f(0)=0$ and $f(1)=b_1-(1-A)=A-a_1>0$.
Finally, we define $\gamma$ by:
\be \gamma(t)=\int_{0}^t\!\!\int_{0}^s f(x)\, dx\, ds.\ee
Then, $\gamma(0)=0$ and $\gamma(1)>0$. Moreover for every $t\in (0,1)$
\be \gamma'(t)=\int_{0}^t f(x)\, dx>0,\ee
since the integrand is non-negative and non-zero on any interval $(0, t)$.

For $t\leq \frac{1}{k}$ we have the following estimate for $|\gamma(t)|$:
\begin{equation}\label{eq1}
\begin{split}
|\gamma(t)| &\leq \int_{0}^{\frac{1}{k}}\!\!\!\int_{0}^{\frac{1}{k}}|f(x)|\, dx\, ds\leq\frac{1}{k^2}  \sup_{t\leq \frac{1}{k}} |f(t)| =\frac{1}{k^2}\! \sup_{\frac{1}{k+1}\leq t\leq \frac{1}{k}} \! |f(t)|\\
&= \frac{1}{k^2}-1+A+(1-a_1-\cdots-a_k) = A-a_1-\cdots-a_k\\
&= \sum_{j\geq k+1} a_j\leq \frac{c_k}{3}.
\end{split}
\end{equation}
Similarly, for $t\leq \frac{1}{k}$ we have the following estimate for the derivative $|\gamma'(t)|$:
\begin{align}
\label{eq2}|\gamma'(t)|&\leq\int_{0}^{\frac{1}{k}}|f(x)|\, dx\leq \sum_{j\geq k}\frac{a_j}{j^2}\leq\bigg( \sum_{j\geq k}a_j^2\bigg)^{\!\frac{1}{2}} \bigg( \sum_{j\geq k}\frac{1}{j^4}\bigg)^{\!\frac{1}{2}}\leq\frac{c_k}{3}.
\end{align}
Moreover, for $t\leq \frac{1}{k}$ we have the following estimate for $\sum_{j=2}^k|\gamma^{(j)}(t)|$:
\begin{equation}\label{eq3}
\begin{split}
\sum_{j=2}^k |\gamma^{(j)}(t)| &= \sum_{j=0}^{k-2}|f^{(j)}(t)|
\leq\sup_{i\geq k}\sup_{t\in [\frac{1}{i+1}, \frac{1}{i}]}\sum_{j=0}^{k-2}|\varphi_i^{(j)}(t)|\\
&= \sup_{t\in [\frac{1}{k+1}, \frac{1}{k}]}\sum_{j=0}^{k-2}|\varphi_k^{(j)}(t)|\leq \frac{c_k}{3}.
\end{split}
\end{equation}
Combining \eqref{eq1}, \eqref{eq2} and \eqref{eq3} we get:
\[\pushQED{\qed}
\sum_{j=0}^k|\gamma^{(j)}(t)|\leq c_k\quad \hbox{ for } t\leq \frac{1}{k}.\qedhere\popQED\]
\renewcommand{\qedsymbol}{}
\end{proof}

\begin{lemma}\label{lemma:inter}Let $M$ be a smooth manifold, and let $X \subsetneq \Int(M)$ be a finite simplicial subcomplex. Let also $U$ be a compact, smooth regular neighborhood of $X$.
Then there exists a $C^{\infty}$ function $\alpha\: U \to [0,1]$ such that:
\begin{enumerate} 
\item $\alpha^{-1}(0)= X$ and $\alpha^{-1}(1)=\partial U$;
\item $\alpha$ has no critical points in $(0,1)$;
\item the sublevels $\subl{\alpha}{\epsilon}$ are smooth regular neighborhoods of $X$ for every $\epsilon\in (0,1]$.
\end{enumerate}
\end{lemma}
\begin{proof}

There is a triangulation of $M$ such that the star $\st(X, M)$ of $X$ is a regular neighborhood of $X$ in $M$ (see Rourke and Sanderson \cite{rourke:PL}). We set $\widetilde U = \st(X, M)$.
Let us also consider the function $\widetilde{f} \: \widetilde{U} \to [0,2]$ which is affine-linear on every simplex and which takes value zero on $X$ and value $2$ on $\partial \widetilde{U}$. Observe that this function is piecewise smooth, and hence Lipschitz. Let $\widetilde{v}$ be the Lipschitz vector field on $\widetilde{U}$ constructed as in  \cite[Lemma $2.3$]{hirsch:combinatorial}. It follows from \cite{hirsch:regular} that there exists a fundamental system of neighborhoods $\{U_n\}_{n\geq 1}$ of $X$ in $M$ such that
\begin{enumerate}
\item  $U_1\subset \widetilde{U}$ and $U_{n+1}\subset \Int(U_n)$ for every $n\geq 1$;
\item $X_n=\partial U_n$ is a smooth embedded submanifold of $M$ of codimension one;
\item the vector field $\widetilde{v}$ is transversal to $X_n$ for every $n\geq 1$;
\item $U_n\setminus \Int(U_{n+1})$ is PL diffeomorphic to $X_n\times [\frac{1}{n+1}, \frac{1}{n}]$.
\end{enumerate}
For every $n\geq 1$ let now $v_n$ be a smooth vector field on $U_n\setminus \Int(U_{n+1})$ which is sufficiently close to $\widetilde{v}$ in order to guarantee that the flow lines of $v_n$ give a diffeomorphism 
\be \varphi_n\: X_n\times \left[\frac{1}{n+1}, \frac{1}{n}\right]\to U_n\setminus \Int(U_{n+1}).\ee
All these diffeomorphisms can be glued together to give a diffeomorphism
\be \varphi\: X_1\times (0, 1]\to U_1\setminus X.\ee 

Let $\pi_2$ be the projection of $X_1\times (0,1]$ on the second factor. 
Notice that, by the uniqueness up to smooth isotopies of regular neighborhoods, we can assume $U=U_1$ and define the map $\alpha\: U \to [0,1]$ such that
\begin{equation}
    \begin{cases}
    \alpha\big\vert_{U\setminus X} = \gamma \circ \pi_2 \circ \varphi^{-1} \\
    \alpha\big\vert_{X} \equiv 0,
    \end{cases}
\end{equation}
where $\gamma \: [0,1] \to [0,1]$ is an appropriate smooth function that we will now construct. Such function will be monotone, with nonzero derivative on $(0, 1]$ and with derivatives that go to zero sufficiently fast at zero, in such a way that $\alpha $ will be smooth and with no critical points in $U$ other than the points in $X$.

For the construction of $\gamma$, denote $\beta=\pi_2\circ \varphi^{-1}$ and fix fiberwise norms on the jet bundles $J^k(U, \R)$ and $J^k(\R, \R)$. We denote both these norms by $\|\cdot \|$. Denoting by $j_p^kf\in J_p(X, M)$ the $k$-th jet of a smooth map $f\: X\to M$, we claim that  for every $k\geq0$ there exists $p(k)>0$ such that every $p\in U\setminus X$ and for every smooth function $\gamma\: [0,1]\to \R$, we have
\be \label{eq:bound1} \|j_p^k(\gamma \circ \beta)\|\leq C(k)\cdot \|j_{\beta(p)}^k\gamma\|\cdot \|j_p^k\beta\|^k.\ee
Denoting by $(x_1, \ldots, x_n)$ coordinates on $U_1$, Fa\'a di Bruno's formula reads:
\be \frac{\partial^k(\gamma\circ \beta)}{\partial x_1\cdots \partial x_{k}}(x(p))=\sum_{\pi\in \mathcal{P}_k}\gamma^{(|\pi|)}(\beta(x))\cdot \prod_{b\in \pi}\frac{\partial^{|b|}\beta}{\partial x_{i_1}\cdots \partial x_{i_{b}}}(x(p)),\ee
where $\mathcal{P}_k$ denotes the set of all partitions $\pi$ of the set $\{1, \ldots, k\}$ and ``$b\in \pi$'' means that the variable $b$ runs through the list of all of the blocks of the partition $\pi$.
In particular, from this formula we deduce that:
\begin{align}\|j_p^k(\gamma \circ \beta)\|&\leq C_1(k)\left(\sup_{j=0, \ldots, k}|\gamma^{(j)}(\beta(p))|\right)\sum_{\pi\in \mathcal{P}_k} \prod_{b\in \pi}\left|\frac{\partial^{|b|}\beta}{\partial x_{i_1}\cdots \partial x_{i_{b}}}(x(p))\right|\\
&\leq  C_2(k)\left(\sup_{j=0, \ldots, k}|\gamma^{(j)}(\beta(p))|\right)\sum_{\pi\in \mathcal{P}_k} \|j_p^{k}\beta\|^{|b|}\\
&\leq C_3(k)\|j_{\beta(p)}^k\gamma\|\cdot (\#\mathcal{P}_k)\cdot \|j_p^k\beta\|^k\\
&\leq C(k)\cdot \|j_{\beta(p)}^k\gamma\|\cdot \|j_p^k\beta\|^k,\end{align}
which is \eqref{eq:bound1}. For every $k\geq 0$ set now
\be \tilde{c}_k:=C(k)\sup_{p\in \varphi(X_1\times [\frac{1}{k+1}, 1])}\|j_p^{k}\beta\|^k.\ee
From Lemma \ref{lemma:smooth} (applied with the choice $\{c_k=(\tilde{c}_kk)^{-1}\}$) we get a smooth function $\tilde{\gamma}\: [0,1]\to \R$ such that $\tilde{\gamma}(0)=0$, $\tilde{\gamma}(1)>0$, $\tilde{\gamma}'(t)>0$ on $(0, 1]$ and $\sum_{j\leq k}|\tilde{\gamma}^{(j)}(t)|\leq \frac{1}{k\tilde{c}_k}$ for all $t\leq \frac{1}{k}.$ 

We define our function as $\gamma:=\tilde{\gamma}/\tilde{\gamma}(1),$ so that $\gamma(1)=1$. We claim that for every $\ell\geq 0$ and for every $\epsilon>0$ there exists $m\geq 0$ such that $\|j_p^\ell (\gamma \circ \beta)\|\leq \epsilon$ for all $p\in \beta^{-1}([0, \frac{1}{m}])$. This means that for every $\ell\geq 0$ the $\ell$-th jet of $\gamma\circ\beta$ goes uniformly to zero as $p$ approaches $X$. Hence, the extension of $\gamma\circ\beta$ to zero on $X$ is smooth on $U_1$.

In order to prove the claim, let $\ell\geq0$, $\epsilon>0$ and pick $m\geq0$ such that $\frac{1}{m}\leq \epsilon\tilde{\gamma}(1)$. Let $p=\varphi(x, t)$ be a point with $t=\beta(p)\leq \frac{1}{m}$ and choose $k\geq m$ such that $\frac{1}{k+1}\leq t\leq \frac{1}{k}$. Then:
\begin{align} \|j_{p}^\ell(\gamma \circ \beta)\|&\leq \|j_{p}^{k}(\gamma \circ \beta)\|\\
&\leq C(k)\cdot \|j_{\beta(p)}^k\gamma\|\cdot \|j_p^k\beta\|^k\\
&=\frac{C(k)}{\tilde{\gamma}(1)}\|j_{\beta(p)}^k\tilde{\gamma}\|\cdot \|j_p^k\beta\|^k\\
&=\frac{C(k)}{\tilde{\gamma}(1)}\left(\sum_{j=0}^k|\tilde{\gamma}^{(j)}(t(p))|\right)\cdot \|j_p^k\beta\|^k\\
&\leq \frac{C(k)}{\tilde{\gamma}(1)}\frac{1}{k\tilde{c}_k}\cdot \|j_p^k\beta\|^k\\
&\leq \frac{1}{k\tilde{\gamma}(1)}\leq\frac{1}{m\tilde{\gamma}(1)}\leq \epsilon,
\end{align}
which is what we wanted. Since $\tilde{\gamma}'(t)>0$ on $(0, 1]$, the same is true for $\gamma'(t)$ and therefore $\alpha$ has no critical points other than the points in $X$.
\end{proof}

\begin{proof}[Proof of Theorem \ref{thm:existence}]
Without loss of generality we may assume that $M$ is connected, since we can construct the Reeb function independently on each connected component. Let $\alpha\: U\to [0,1]$ be given by Lemma \ref{lemma:inter}, where $U$ is a smooth regular neighborhood of $X$ in $M$.
Let now $V=M\setminus \Int(U)$ and consider a spine $Y$ for $V$. Using Lemma \ref{lemma:inter} we get a smooth regular neighborhood $V'$ of $Y$ in $M$ and a function $\alpha'\: V'\to [0,1]$ with the set of minima $Y$; using the uniqueness of regular neighborhoods, $V$ is diffeomorphic to $V'$ and hence the function $\alpha'$ can be defined on $V$ itself.  Set $\beta=1-\alpha'$. We can glue the functions $\alpha $ and $\beta$ along the common boundary $\partial U = \partial V$, using Lemma \ref{remark:glue}. The result is a Reeb function $f$ on $M$ with extrema $X$ and $Y$, as desired.
\end{proof}

\subsection{Relation between maxima and minima} 
\begin{lemma}\label{lemma:h}
Let $M$ be a smooth, connected, compact manifold and let $f\: M\to [0,1]$ be a Reeb function such that $X_0=f^{-1}(0)$ is a finite simplicial subcomplex. For every $\epsilon\in (0,1)$ there exists  a smooth regular neighborhood $T$ of $X_0$ in $M$ such that $T\subset \Int(B_\epsilon(f))$ and  $W=B_\epsilon(f)\setminus \Int(T)$ is an h-cobordism between $\partial T$ and $f^{-1}(\epsilon)=\partial B_\epsilon(f).$ In particular  $M\setminus f^{-1}(1)$ deformation retracts to $X_0$.
\end{lemma}
\begin{proof}
Without loss of generality we assume that $X_0$ is connected; the non-connected case can be treated analogously, working separately on every connected component.

We observe first that the sublevels $\subl{f}{\epsilon}$ are also connected. In fact, suppose by contradiction that there exists $0<a<1$ such that $\subl{f}{a}$ is not connected. Denote by $C$ a connected component of $\subl{f}{a}$ that does not contain $X_0$. Let
\begin{equation}
b = \min_{x\in C} f(x)
\end{equation}
and fix $\bar{x}\in C$ such that $f(\bar{x}) = b$. This is a critical point for $f$ that is not in $X_0$ or $X_1$, so we get the contradiction.

Moreover the family of sublevels $\{ \subl{f}{\epsilon} \}_{\epsilon>0}$ is a fundamental system of neighborhoods. In fact let $V$ be a closed tubular neighborhood of $X_0$ and take
\begin{equation}
s = \min_{x\in \partial V} f(x).
\end{equation}
If we assume that $ \subl{f}{\frac{s}{2}} \nsubseteq \Int V$, then we can separate it as
\begin{equation}
\left(  \subl{f}{\frac{s}{2}}  \cap \Int V \right) \sqcup \left(  \subl{f}{\frac{s}{2}}  \cap (M\setminus V) \right),
\end{equation}
which leads to a contradiction, since $\subl{f}{\frac{s}{2}}$ is connected.

Let $\alpha\: U\to [0,1]$ be given by Lemma \ref{lemma:inter} and consider the family $\{B_{\epsilon}(\alpha) \}_{0<\epsilon\leq 1}$ of smooth regular neighborhoods of $X_0$ (recall that $U=B_1(\alpha)$).
Let $0<\epsilon_3<\epsilon_2<\epsilon_1<\epsilon$ be such that 
\begin{equation}
 B_{\epsilon_3}(\alpha)\subset \Int(B_{\epsilon_2}(f))\subset B_{\epsilon_2}(f)\subset \Int(B_{\epsilon_1}(\alpha))\subset B_{\epsilon_1}(\alpha)\subset \Int(B_\epsilon(f)).
 \end{equation}
 We set $T=B_{\epsilon_3}(\alpha)$ and fix a Riemannian metric on $M$, so that we can consider the gradients of $f$ and $\alpha$ (and their flows). We can write $W=W'\cup f^{-1}([\epsilon_2, \epsilon])$ with $W'=B_{\epsilon_2}(f)\setminus \Int(T).$  Following the flow of $-\nabla f$ one first deformation retracts $W$ onto $W'$. Since $W'\subset \alpha^{-1}([\epsilon_3, \epsilon_1])\subset W$, we can restrict  to $W'$ the deformation retraction of $\alpha^{-1}([\epsilon_3, \epsilon_1])$ onto $\partial T=\alpha^{-1}(\epsilon_3)$, given by the flow of $-\nabla \alpha$. This gives a deformation retraction of $W$ onto $\partial T$. 
 A similar reasoning, exchanging the roles of $\alpha$ and $f$ and using their gradients, yields a deformation retraction of $W$ onto $f^{-1}(\epsilon)$. 
This proves the first part of the lemma. 

The second part of the lemma follows easily: $M\setminus f^{-1}(1)$ deformation retracts to $B_\epsilon(f)$ using the flow of $-\nabla f$; $B_\epsilon(f)=W\cup T$ deformation retracts to $T$ (by the previous part) and $T$ deformation retracts to $X_0$, since it is a regular neighborhood.
\end{proof}

\begin{remark}\label{remark:homology} Suppose that $M$ is a closed oriented $n$-manifold and let $f \: M \to [0,1]$ be a Reeb function having as extrema two simplicial subcomplexes $X_0, X_1 \subset M$. There are isomorphisms $H_i(M,X_1) \cong H_i(M, M-X_0)\cong H^{n-i}(X_0)$, for all $i \geq 0$, the former being induced by inclusion and using Lemma \ref{lemma:h}, while the latter is a well-known duality (see for example \cite[Proposition 3.46]{hatcher:at}). Then, the long exact sequence of the pair $(M, X_1)$ yields the following long exact sequence
\[\cdots \to H_i(X_1) \to H_i(M) \to H^{n-i}(X_0) \to H_{i-1}(X_1) \to \cdots.\]
In the non-orientable case an analogous long exact sequence holds with $\Z_2$ coefficients.

This means that the image of the inclusion $H_{i} \left( X_1 \right) \to H_i \left( M \right)$ lies in the orthogonal complement of the restriction $H^{n-i} \left( M \right) \to H^{n-i}\left(  X_0 \right)$. This gives a necessary condition for two subcomplexes of $M$ to be the extrema of a Reeb function.

In the particular case of $M= S^n$ we can notice something more. Alexander duality in fact states that $\widetilde{H}_i(S^n \setminus X_0) \cong \widetilde{H}^{n-i-1}(X_0)$ and therefore
\begin{equation}
\widetilde{H}_i(X_1) \cong \widetilde{H}^{n-i-1}(X_0).
\end{equation}
This implies for example that the sum of the Betti numbers of $X_0$ must coincide with the sum of the Betti numbers of $X_1$.
\end{remark}

\section{Dimension 3}
We prove now Theorem \ref{thm_g_infty} from Section \ref{sec:intro}, whose statement we recall here for the reader's convenience.
\begin{theorem*}
Let $M$ be a smooth closed connected $3$-manifold. Then the following conditions are equivalent:
\begin{enumerate}
\item there exists a Reeb function $f \colon M \to [0,1]$  having connected 1-dimensional subcomplexes of the same Euler characteristic $\chi$ as extrema, for some $\chi\leq 1$;
\item the Heegaard genus of $M$ is at most $1-\chi$ in the orientable case and $2-2\chi$ in the non-orientable case.
\end{enumerate} 
\end{theorem*}

\begin{proof}
$(1)\Rightarrow (2)$. Given $\epsilon>0$, by Lemma \ref{lemma:h} there exists a regular neighborhood $T$ of $X_0$ such that  $T\subset \Int(B_\epsilon(f))$ and  $W=B_{\epsilon}(f)\setminus \Int(T)$ is an h-cobordism. By \cite[Theorem 10.2]{Hem76}, it follows that this cobordism is a product, that is $W\cong \partial T\times [0,1]$. Since $T$ is a genus-$g$ handlebody,  with $g=1-\chi$ in the orientable case and $g=2-2\chi$ in the non-orientable case, it follows that $B_\epsilon(f)=T\cup W$ is a genus-$g$ handlebody. Similarly, the superlevel $A_{\epsilon}(f)$ is a genus-$g$ handlebody and the Heegaard genus of $M=B_{\epsilon}(f)\cup_{\partial}A_{\epsilon}(f)$ is at most $g$.

$(2)\Rightarrow (1)$. Consider a genus $g$ Heegaard splitting of $M=P\cup_\partial Q$, where $P$ and $Q$ are $3$-dimensional handlebodies of genus $g$, with $g=1-\chi$ in the orientable case and $g=2-2\chi$ in the non-orientable case. Then there are connected graphs $X_0\subset P$ and $X_1\subset Q$ with Euler characteristic $\chi$, such that $P$ is a regular neighborhood of $X_0$ and $Q$ is a regular neighborhood of $X_1$. The proof of the claim follows immediately from Lemma \ref{lemma:inter} and the uniqueness of regular neihgborhoods, as in the proof of Theorem~\ref{thm:existence}.
\end{proof}
\begin{remark}In the proof of the first implication of the previous theorem we used \cite[Theorem 10.2]{Hem76}. The application of this result does not need the Poincar\'e Conjecture because the sublevel $B_{\epsilon}(f)$ can be embedded into a handlebody (hence in $\R^3$) following the flow of $-\nabla f$. In particular $B_\epsilon(f)$ does not contain fake $3$-cells.
\end{remark}

\begin{remark}\label{remark:02} If $M$ is a closed, orientable $3$-manifold and $f\: M\to [0,1]$ is a Reeb function such that $X_0\cong X_1\cong S^k$ are smoothly embedded $k$-spheres, then one can prove that $M\cong S^3\sqcup S^3$, if $k=0$,  and $M\cong S^1\times S^2$ if $k=2.$ The case $k=0$ follows from Milnor's version of Reeb's Sphere Theorem. The case $k=2$ is also elementary, following the same lines as Proposition \ref{propo:sei}, and using the fact that the h-cobordism holds true also in dimension $n=3$, being equivalent to the Poincar\'e Conjecture.\end{remark}

\section{Higher dimensions}\label{sec_>5}

Consider a manifold $M$ of dimension $n$ and a Reeb function with two smoothly embedded copies of $S^k$, $k<n$, as extrema. If we restrict to $n\geq 6$, we may use cobordism techniques and prove the following statements.

\begin{proposition}\label{propo:sei}
Let $M$ be a smooth closed connected orientable manifold of dimension $n\geq 6$. Suppose that there exists a Reeb function $f \colon M \to [0,1]$ whose extrema are finite simplicial subcomplexes with fundamental group having trivial Whitehead group. Then $B_\epsilon(f)$ and $A_\epsilon(f)$ are smooth regular neighborhoods of $X_0$ and $X_1$, respectively.
\end{proposition}
\begin{proof}
This follows from Lemma \ref{lemma:h} together with the s-cobordism Theorem \cite{milnor:whitehead} applied to smooth regular neighborhoods of $X_0$ and $X_1$, respectively (see also the proof of Theorem \ref{thm_g_infty}).
\end{proof}

From this result, we can then deduce the following corollary.

\begin{corollary}\label{cor:main_higher_dimension}
Let $M$ be a smooth closed connected orientable manifold of dimension $n\geq 6$ and let $1 \leq k < n$. Then the following conditions are equivalent:
\begin{enumerate}
    \item there exists a Reeb function $f\: M\to [0,1]$ with two smoothly embedded $k$-spheres as extrema, with trivial normal bundle;
    \item $M$ is obtained by gluing two copies of $S^k\times D^{n-k}$ along their boundaries with a diffeomorphism of $S^k\times S^{n-k-1}=\partial (S^k\times D^{n-k})$.
\end{enumerate}
\end{corollary}

Notice in particular that for $k=n-1$, $M$ is diffeomorphic to $S^{n-1}\times S^1$.

\section{Concluding remarks and open questions}

\begin{remark} The s-cobordism theorem still holds true in dimension $5$ topologically, so Proposition \ref{propo:sei} and Corollary \ref{cor:main_higher_dimension} can be extended to dimension $5$ up to homeomorphisms instead of diffeomorphisms. However $5$-dimensional s-cobordism theorem fails in the smooth category. Indeed, any two homeomorphic but not diffeomorphic closed smooth simply connected 4-manifolds are known to be smoothly h-cobordant).

In dimension $4$, the h-cobordism theorem holds topologically in the simply connected case, while in the smooth category it is equivalent to the 4-dimensional Smooth Poincaré Conjecture, which is still open. It would be intriguing to investigate an analogous to Corollary \ref{cor:main_higher_dimension} in dimension $4$, in the smooth category.
\end{remark}

\begin{remark}\label{product/rmk}
If $f\: M \to [0,1]$ is a Reeb function with extrema $X_0, X_1 \subset M$, the flow of $\grad f$, with respect to a fixed Riemannian metric, yields a diffeomorphism \[M \setminus (X_0\cup X_1) \cong f^{-1}(\epsilon) \times \R,\] for any $\epsilon \in (0,1)$, and $f^{-1}(\epsilon)$ is a smooth closed hypersurface in $M$. In particular, $M \setminus (X_0\cup X_1)$ has finitely many ends, because the connected components of $X_0\cup X_1$ can be identified with the end points of $M \setminus (X_0\cup X_1)$, which, in turn, correspond to the connected components of $f^{-1}(\epsilon)$. Moreover, if $M$ is connected and $X_0$ and $X_1$ have topological dimension at most $\dim M -2$, then they do not disconnect $M$ and hence they must be connected. Therefore, certain subspaces of $M$ (for example, a Cantor set) cannot occur as extrema of a Reeb function.
\end{remark}

So, the following question arises naturally.
\begin{question}
Which subspaces of a closed manifold can be realized as one of the extrema of a certain Reeb function?
\end{question}

Another question is motivated by our constructions in Lemma \ref{lemma:inter} and in Theorem \ref{thm:existence}, which produce Reeb functions admitting a collapsing pseudo-gradient vector field $v$, that is, the flow lines of $v$ (resp. $-v$) give a collapsing of $M\setminus X_0$ to $X_1$ (resp. of $M\setminus X_1$ to $X_0$).

\begin{question}
Does there exist a closed manifold $M$ with a Reeb function $f\: M \to [0,1]$ with no collapsing pseudo-gradient vector field?
\end{question}


\end{document}